\documentclass[12pt]{amsart}

\usepackage{amsmath,amssymb,amsbsy,amsfonts,amsthm,latexsym,
                     amsopn,amstext,amsxtra,euscript,amscd,mathrsfs}
\usepackage{amsmath,amssymb,amsbsy,amsfonts,latexsym,amsopn,amstext,cite,
                                               amsxtra,euscript,amscd,bm}
\usepackage{mathtools}
\usepackage{todonotes}
\usepackage{url}
\usepackage[colorlinks,linkcolor=blue,anchorcolor=blue,citecolor=blue,backref=page]{hyperref}

\renewcommand*{\backref}[1]{}
\renewcommand*{\backrefalt}[4]{%
    \ifcase #1 (Not cited.)%
    \or        (p.\,#2)%
    \else      (pp.\,#2)%
    \fi}

\begin{document}




\newfont{\teneufm}{eufm10}
\newfont{\seveneufm}{eufm7}
\newfont{\fiveeufm}{eufm5}
%
%
\newfam\eufmfam
 \textfont\eufmfam=\teneufm \scriptfont\eufmfam=\seveneufm
 \scriptscriptfont\eufmfam=\fiveeufm
%
%
\def\frak#1{{\fam\eufmfam\relax#1}}
%


\def\bbbr{{\rm I\!R}} 
\def\bbbm{{\rm I\!M}}
\def\bbbn{{\rm I\!N}} 
\def\bbbf{{\rm I\!F}}
\def\bbbh{{\rm I\!H}}
\def\bbbk{{\rm I\!K}}
\def\bbbp{{\rm I\!P}}
\def\bbbone{{\mathchoice {\rm 1\mskip-4mu l} {\rm 1\mskip-4mu l}
{\rm 1\mskip-4.5mu l} {\rm 1\mskip-5mu l}}}
\def\bbbc{{\mathchoice {\setbox0=\hbox{$\displaystyle\rm C$}\hbox{\hbox
to0pt{\kern0.4\wd0\vrule height0.9\ht0\hss}\box0}}
{\setbox0=\hbox{$\textstyle\rm C$}\hbox{\hbox
to0pt{\kern0.4\wd0\vrule height0.9\ht0\hss}\box0}}
{\setbox0=\hbox{$\scriptstyle\rm C$}\hbox{\hbox
to0pt{\kern0.4\wd0\vrule height0.9\ht0\hss}\box0}}
{\setbox0=\hbox{$\scriptscriptstyle\rm C$}\hbox{\hbox
to0pt{\kern0.4\wd0\vrule height0.9\ht0\hss}\box0}}}}
\def\bbbq{{\mathchoice {\setbox0=\hbox{$\displaystyle\rm
Q$}\hbox{\raise
0.15\ht0\hbox to0pt{\kern0.4\wd0\vrule height0.8\ht0\hss}\box0}}
{\setbox0=\hbox{$\textstyle\rm Q$}\hbox{\raise
0.15\ht0\hbox to0pt{\kern0.4\wd0\vrule height0.8\ht0\hss}\box0}}
{\setbox0=\hbox{$\scriptstyle\rm Q$}\hbox{\raise
0.15\ht0\hbox to0pt{\kern0.4\wd0\vrule height0.7\ht0\hss}\box0}}
{\setbox0=\hbox{$\scriptscriptstyle\rm Q$}\hbox{\raise
0.15\ht0\hbox to0pt{\kern0.4\wd0\vrule height0.7\ht0\hss}\box0}}}}
\def\bbbt{{\mathchoice {\setbox0=\hbox{$\displaystyle\rm
T$}\hbox{\hbox to0pt{\kern0.3\wd0\vrule height0.9\ht0\hss}\box0}}
{\setbox0=\hbox{$\textstyle\rm T$}\hbox{\hbox
to0pt{\kern0.3\wd0\vrule height0.9\ht0\hss}\box0}}
{\setbox0=\hbox{$\scriptstyle\rm T$}\hbox{\hbox
to0pt{\kern0.3\wd0\vrule height0.9\ht0\hss}\box0}}
{\setbox0=\hbox{$\scriptscriptstyle\rm T$}\hbox{\hbox
to0pt{\kern0.3\wd0\vrule height0.9\ht0\hss}\box0}}}}
\def\bbbs{{\mathchoice
{\setbox0=\hbox{$\displaystyle     \rm S$}\hbox{\raise0.5\ht0\hbox
to0pt{\kern0.35\wd0\vrule height0.45\ht0\hss}\hbox
to0pt{\kern0.55\wd0\vrule height0.5\ht0\hss}\box0}}
{\setbox0=\hbox{$\textstyle        \rm S$}\hbox{\raise0.5\ht0\hbox
to0pt{\kern0.35\wd0\vrule height0.45\ht0\hss}\hbox
to0pt{\kern0.55\wd0\vrule height0.5\ht0\hss}\box0}}
{\setbox0=\hbox{$\scriptstyle      \rm S$}\hbox{\raise0.5\ht0\hbox
to0pt{\kern0.35\wd0\vrule height0.45\ht0\hss}\raise0.05\ht0\hbox
to0pt{\kern0.5\wd0\vrule height0.45\ht0\hss}\box0}}
{\setbox0=\hbox{$\scriptscriptstyle\rm S$}\hbox{\raise0.5\ht0\hbox
to0pt{\kern0.4\wd0\vrule height0.45\ht0\hss}\raise0.05\ht0\hbox
to0pt{\kern0.55\wd0\vrule height0.45\ht0\hss}\box0}}}}
\def\bbbz{{\mathchoice {\hbox{$\sf\textstyle Z\kern-0.4em Z$}}
{\hbox{$\sf\textstyle Z\kern-0.4em Z$}}
{\hbox{$\sf\scriptstyle Z\kern-0.3em Z$}}
{\hbox{$\sf\scriptscriptstyle Z\kern-0.2em Z$}}}}
\def\ts{\thinspace}

\newtheorem{thm}{Theorem}
\newtheorem{lem}{Lemma}
\newtheorem{lemma}[thm]{Lemma}
\newtheorem{prop}{Proposition}
\newtheorem{proposition}[thm]{Proposition}
\newtheorem{theorem}[thm]{Theorem}
\newtheorem{cor}[thm]{Corollary}
\newtheorem{corollary}[thm]{Corollary}
\newtheorem{rem}[thm]{Remark}

\newtheorem{prob}{Problem}
\newtheorem{problem}[prob]{Problem}
\newtheorem{ques}{Question}
\newtheorem{question}[ques]{Question}


\numberwithin{equation}{section}
\numberwithin{thm}{section}

\def\squareforqed{\hbox{\rlap{$\sqcap$}$\sqcup$}}
\def\qed{\ifmmode\squareforqed\else{\unskip\nobreak\hfil
\penalty50\hskip1em\null\nobreak\hfil\squareforqed
\parfillskip=0pt\finalhyphendemerits=0\endgraf}\fi}

\def\fF{\EuScript{F}}

\def\cA{{\mathcal A}}
\def\cB{{\mathcal B}}
\def\cC{{\mathcal C}}
\def\cD{{\mathcal D}}
\def\cE{{\mathcal E}}
\def\cF{{\mathcal F}}
\def\cG{{\mathcal G}}
\def\cH{{\mathcal H}}
\def\cI{{\mathcal I}}
\def\cJ{{\mathcal J}}
\def\cK{{\mathcal K}}
\def\cL{{\mathcal L}}
\def\cM{{\mathcal M}}
\def\cN{{\mathcal N}}
\def\cO{{\mathcal O}}
\def\cP{{\mathcal P}}
\def\cQ{{\mathcal Q}}
\def\cR{{\mathcal R}}
\def\cS{{\mathcal S}}
\def\cT{{\mathcal T}}
\def\cU{{\mathcal U}}
\def\cV{{\mathcal V}}
\def\cW{{\mathcal W}}
\def\cX{{\mathcal X}}
\def\cY{{\mathcal Y}}
\def\cZ{{\mathcal Z}}
\def\pgdc{\textrm{gcd}}
\newcommand{\rmod}[1]{\: \mbox{mod} \: #1}

\def\Nm{{\mathrm{Nm}}}

\def\Tr{{\mathrm{Tr}}}

\def\epp{\mathbf{e}_{p-1}}

\def\ind{\mathop{\mathrm{ind}}}

\def\mand{\qquad \mbox{and} \qquad}

\def\M{\mathsf {M}}
\def\T{\mathsf {T}}

\newcommand{\commI}[1]{\marginpar{%
\begin{color}{magenta}
\vskip-\baselineskip 
\raggedright\footnotesize
\itshape\hrule \smallskip I: #1\par\smallskip\hrule\end{color}}}

\newcommand{\commM}[1]{\marginpar{%
\begin{color}{blue}
\vskip-\baselineskip 
\raggedright\footnotesize
\itshape\hrule \smallskip M: #1\par\smallskip\hrule\end{color}}}



\newcommand{\ignore}[1]{}

\hyphenation{re-pub-lished}

\parskip 1.5 mm

\def\GL{\operatorname{GL}}
\def\SL{\operatorname{SL}}
\def\PGL{\operatorname{PGL}}
\def\PSL{\operatorname{PSL}}
\def\li{\operatorname{li}}

\def\vec#1{\mathbf{#1}}

\def \F{{\mathbb F}}
\def \K{{\mathbb K}}
\def \Z{{\mathbb Z}}
\def \N{{\mathbb N}}
\def \Q{{\mathbb Q}}
\def \C {{\mathbb C}}
\def \R{{\mathbb R}}
\def\Fp{\F_p}
\def \fp{\Fp^*}

\def \Rc{{\mathcal R}}
\def \Qc{{\mathcal Q}}
\def \Ec{{\mathcal E}}

\def \DN{D_N}
\def\va{\mbox{\bf{a}}}

\def\Kc{\,{\mathcal K}}
\def\Ic{\,{\mathcal I}}

\def\\{\cr}
\def\({\left(}
\def\){\right)}
\def\fl#1{\left\lfloor#1\right\rfloor}
\def\rf#1{\left\lceil#1\right\rceil}

\def\Ln#1{\mbox{\rm {Ln}}\,#1}

\def \nd {\, | \hspace{-1.2mm}/\,}

 \def\e{\mathbf{e}}

\def\ep{\mathbf{e}_p}
\def\eq{\mathbf{e}_q}

\def\wt#1{\mbox{\rm {wt}}\,#1}

\def\Mob{M{\"o}bius }



\title[Algebraic trace functions
by  arithmetic   functions]{Sums of  algebraic trace functions  twisted
by arithmetic   functions}

\author{Maxim A. Korolev}
\address{Steklov Mathematical Institute of Russian Academy of Sciences, ul. Gubkina 8, Moscow, 119991 Russia}
\email{korolevma@mi.ras.ru}

\author{Igor E. Shparlinski}
\address{School of Mathematics and Statistics, University of New South Wales,
 Sydney NSW 2052, Australia}
\email{igor.shparlinski@unsw.edu.au}


\date{}

\pagenumbering{arabic}

\begin{abstract}
We obtain new bounds for short sums of isotypic trace functions
associated to some sheaf modulo prime $p$ of bounded conductor,  
 twisted by  the \Mob function and also by the  generalised divisor function. 
 These trace functions include Kloosterman sums and several other classical number 
theoretic objects.
Our bounds are nontrivial for intervals of length at least 
$p^{1/2+\varepsilon}$ with an arbitrary fixed $\varepsilon >0$, which is 
shorter than the length at least $p^{3/4+\varepsilon}$ in  the case of the  \Mob function 
and at least $p^{2/3+\varepsilon}$  in  the case of the divisor function required in recent 
results of  {\'E}.~Fouvry,  E.~Kowalski and P. ~Michel (2014) and 
E.~Kowalski, P. ~Michel and W.~Sawin (2018), respectively.
\end{abstract}


\maketitle


\section{Background and motivation}
For a prime $p$ and arbitrary integers $m$ and $n$, we define the $s$-dimensional 
Kloosterman sums
$$
K_{s,p}(n)= \sum_{\substack{x_1, \ldots, x_s =1\\ 
x_1 \cdots x_s \equiv n \bmod  p}}^{p-1} \e\left(\frac{x_1 + \ldots + x_s }{p}\right),
$$
where
for a real $z$ we denote
$$
\e(z)=e^{2\pi i z}.
$$

The classical Deligne bound yields the estimate
\begin{equation}
\label{eq:Deligne}
|K_{s,p}(n)| \le s p^{(s-1)/2},
\end{equation}
see~\cite[Equation~(11.58)]{IwKow} and the follow-up discussion.

Since the bound~\eqref{eq:Deligne} is essentially optimal, it is natural to
study cancellations between Kloosterman sum in various families of pairs
$(n,p)$ of parameters, with and without some weights attached. 

Studying cancellations for fixed $m$ and $n$ and varying $n$ is related to the
Linnik conjecture~\cite{Lin}, and thus to  the groundbreaking results of  Kuznetsov~\cite{Kuz},
see also~\cite[Chapter~16]{IwKow}. We refer to~\cite{AnDu,BlMil,Drap,Kir, SarTsim} for some recent
developments and applications.

Recently, the dual question about cancellations between Kloosterman sums $K_s(n;p)$, and more general functions (see below) has
attracted quite a lot of attention due to its applications to several other problems,
see~\cite{BFKMM1,BFKMM2,FKM1,FKM2,FKMRRS,FMRS,KMS1,KMS2,LSZ1,LSZ2,Shp,ShpZha,WuXi,Xi-FKM} and references therein.

Furthermore, it turns out that Kloosterman sums are representatives of a much richer 
class of {\it  isotypic trace functions $\cK(n)$\/} which are associated with  isotypic trace sheaves $\fF$ modulo $p$ of bounded conductor, 
we refer to~\cite{FKM1,FKM1.5} for precise definitions and properties of trace functions. 

For the purpose of this work we do not need to 
know any  specific deep properties of trace functions, it is quite enough to use the facts
which are summarised, for example, in~\cite{FKM1,FKM1.5}.
We also note that this class of functions includes
\begin{itemize}
\item normalised Kloosterman sums $ p^{-(s-1)/2} K_{s,p}(n)$;
\item traces of Frobenius of elliptic curves modulo $p$;
\item exponential functions of the form $\e\(\psi(n)/p\)$
with a rational function $\psi(Z) \in \Q(Z)$,  and similar values of multiplicative 
characters $\chi\(\psi(n)\)$, 
as well as their products (excluding for the exceptional function $\e\left(an/p\right) \chi(n)$ with $a\in \Z$);
\end{itemize}
see, for example,~\cite[Remark~1.4]{FKM1}.


Here, for a positive   integer $N$, a prime  $p$, and  an isotypic trace function $\cK(n)$ modulo  $p$
 we consider the sum
\begin{equation}
\label{eq: sum M}
\M_{p}(\cK, N) =
\sum_{n  \leq N}\mu(n) \cK(n)
\end{equation}
with the \Mob function, which is given by
$\mu(n)=0$ if an integer $m$ is divisible by a
prime square and $\mu(n)=(-1)^r$ if $m$ is a product
of $r$ distinct primes.

We note that some of the motivation to consider the sums~\eqref{eq: sum M}
comes from the program devised by Sarnak~\cite{Sarn}  to
establish  the pseudorandomness of the \Mob function
and in particular, to show that it is not correlated with other arithmetic sequences.

In particular, Fouvry,  Kowalski and Michel~\cite[Theorem~1.7]{FKM1}, have  given a bound on $\M_{p}(\cK, N)$ 
for a wide class of isotypic trace functions $\cK(n)$ of bounded conductor
with a power saving against the trivial bound
$$
\M_{p}(\cK, N) = O\(N\)
$$
assuming that
\begin{equation}
\label{eq: old range N}
N \ge p^{3/4 + \varepsilon}
\end{equation}
for some fixed $\varepsilon > 0$.
Furthermore, it is natural to expect  that~\cite[Theorem~1.8]{BFKMM2} can be extended to the sums~\eqref{eq: sum M}
and thus improve the bound of~\cite[Theorem~1.7]{FKM1}, however it is not likely
to extend the range~\eqref{eq: old range N}.

Here we use a different approach to obtain a nontrivial bound
on the sums~\eqref{eq: sum M} in a  shorter range
\begin{equation}
\label{eq: new range N}
 N \ge p^{1/2 + \varepsilon}.
\end{equation}
On the other hand, the saving now is only logarithmic. We remark that Fouvry,  Kowalski and Michel~\cite[Remark~1.9]{FKM1} mention the possibility of a nontrivial
bound in the range~\eqref{eq: new range N} via the method of Bourgain, Sarnak
and Ziegler~\cite[Theorem~2]{BSZ}. We however 
obtain an explicit bound on the saving, which seems to be stronger than the one
achievable via the approach of~\cite{BSZ}. More precisely, using a version of~\cite[Theorem~2]{BSZ}
seem to lead to a saving which is about a square-root of our saving.

Furthermore, for a fixed integer $\nu\ge 1$ we also consider the sums
$$
\T_{p,\nu}(\cK,N)= \sum_{n\le N}\tau_\nu(n)\cK(n).
$$
with the  generalised divisor function  $\tau_\nu(n)$, which is defined as the number
of ordered representations $n = d_1\ldots d_\nu$ with  integer numbers $d_1, \ldots, d_\nu \ge 1$.
Using the bound
\begin{equation}
\label{eq:sum tau-k}
\sum_{n\le z}\tau_\nu(n) =O\(z (\log z)^{\nu-1}\)
\end{equation}
for any real $z \ge 2$,  see~\cite[Equation~(1.80)]{IwKow}, we see that we have the
$$
\T_{p,\nu}(\cK,N) = O\(N(\log N)^{\nu-1}\),
$$
with the implied constant which depends only on $\nu$.

As in the case of $\M_{p}(\cK, N) $, we also give nontrivial bounds on the sums $\T_{p,\nu}(\cK,N)$ under
the condition~\eqref{eq: new range N}. We remark that our treatment of the sums
$\T_{p,\nu}(\cK,N)$ follows the same pattern as for the sums  $\M_{p}(\cK, N) $ but is more
involved in the case of the divisor function since only its average values admit good
bounds (see~\eqref{eq:sum tau-k} or~\eqref{eq:sum tau-k-2} below),
while individual values can be quite large.

In the case of  $\nu=2$ and normilised $s$-dimensional Kloosterman sums $\cK(n) =  p^{-(s-1)/2} K_{s,p}(n)$,  a 
nontrivial bound on  $\T_{p,2}(\cK,N) = O\(N \log N \)$
has recently been given by Kowalski, Michel and Sawin~\cite{KMS2}
under the condition
$$
N  \ge p^{2/3 + \varepsilon}.
$$
Here we extend this range to~\eqref{eq: new range N}.

We also remark, that at least in the case of Kloosterman sums,  once can use the results and 
ideas of  Liu, Shparlinski and Zhang~\cite{LSZ2}  to obtain  a power saving bounds for 
analogues of  $\M_{p}(\cK, N)$ and $\T_{p,\nu}(\cK, N)$ modulo prime powers  $q=p^k$,
 in a much shorter range, namely, for  $N \ge q^{\varepsilon}$.

\section{Our approach and main results}

%

To estimate  $\M_{p}(\cK, N) $ and $\T_{p,\nu}(\cK,N)$  we employ the method of~\cite{Kor2}, also used in~\cite{GJK}.
This is then combined with some  results  of
Fouvry,  Kowalski and Michel~\cite{FKM1},  see Lemma~\ref{lem:sum K compl}
below.

Before we formulate our results we need to recall that  the notations  $F \ll G$ and $F = O(G)$, are equivalent
to $|F|  \le c G)$ for some constant $c>0$,
which throughout the paper may occasionally depend
on the integer parameters  $\nu$ and the conductor of the trace function $\cK$. 

Following~\cite{FKM1}, we say that $\cK(n)$ is an {\it non-exceptional\/} function modulo $p$ ) if it is not proportional to a function
of the form  $\e\left(an/p\right)\chi(n)$ with an integer $a$ and a multiplicative character $\chi$ modulo $p$. 

\begin{theorem}\label{thm:Sum M}
For any fixed real $\varepsilon>0$, if a prime $p$ and an integer $N$
satisfy~\eqref{eq: new range N},  then for any  non-exceptional  isotypic trace function  $\cK$
 associated to some sheaf $\fF$ modulo $p$ of  bounded conductor, 
we have 
$$
\M_{p}(\cK, N)  \ll  \varepsilon^{-1} N  \frac{ \log \log p}{\log  p}.
$$
\end{theorem}

We also have a similar result for $\T_{p,\nu}(\cK,N)$. In this case we however need to add the 
co-primality restriction  $\gcd(a,p)=1$.

\begin{theorem}\label{thm:Sum T}
For any fixed integer $\nu \ge 2$ and real $\varepsilon>0$, if a prime $p$ and an integer $N$
satisfy~\eqref{eq: new range N},  then for any  non-exceptional   isotypic trace function  $\cK$
 associated to some sheaf $\fF$ modulo $p$ of  bounded conductor, 
we have 
$$
\T_{p,\nu}(\cK,N)\ll \varepsilon^{-\nu}N\,\frac{(\log\log{p})^{\nu}}{\log{p}}. 
$$
\end{theorem}


 \section{Correlations of trace functions}

We    recall the following bound  which is combination of~\cite[Proposition~6.2]{FKM1} 
with~\cite[Theorem~6.3]{FKM1}, see also~\cite{FKM2, FKMRRS,FMRS} for several 
other variations of this result.

\begin{lemma}
\label{lem:sum K compl} 
For any  non-exceptional  isotypic trace function  $\cK$
associated to some sheaf $\fF$ modulo a prime $p$ of    bounded conductor, 
there exists a set $\cE_\fF \subseteq\F_p$  of cardinality $\#  \cE_\fF \ll 1$, 
such that uniformly over $a  \in \F_p \setminus \cE_\fF$ we have 
$$
 \sum_{n=1}^p  \cK(n)  \overline{\cK(an)} \e(hn/p) \ll p^{1/2} .
$$
\end{lemma}

Now using a standard reduction between complete and incomplete sums,
see~\cite[Section~12.2]{IwKow}, we immediately derive
a bound  on a sum of products of    $\cK(an)\cK(bn)$, $i=1, \ldots, \nu$,
with $n$ running over an interval $[1, N]$.

\begin{cor}
\label{cor:sum K compl} 
For any  non-exceptional   isotypic trace function  $\cK$
associated to some sheaf $\fF$ modulo a prime $p$  of    bounded conductor, 
there exists a set $\cE_\fF \subseteq\F_p$  of cardinality $\#  \cE_\fF \ll 1$, 
such that uniformly over $a, b \in \F_p$ with $a/b \not \in \F_p \setminus \cE_\fF$  and
an arbitrary integer  $N\le p$ we have
$$
 \sum_{n=1}^N \cK(an)\overline{\cK(bn) } \ll p^{1/2} \log p.
$$
\end{cor}

 \section{Integers avoiding some prime divisors}

Give two real number $y\ge x> 2$ and an integer $N \ge 1$ we denote
by $\cA_0(N;x,y)$ the set  of positive integers $n \le N$ that do not have
a prime divisor in the  half-open interval $(x,y]$. We need the following upper bound
on the cardinality $\# \cA_0(N;x,y)$, 
which follows instantly from the so called {\it Fundamental Lemma\/} of combinatorial sieve,
see, for example,~\cite[Lemma~6.8]{FrIw} or~\cite[Part~I, Theorem~4.4]{Ten}.

\begin{lemma}
\label{lem:ANxy} Uniformly over  integers $N$ and real $x$ and $y$ with  $N \ge y \ge x \ge 2$,  we have
$$
\# \cA_0(N;x,y) \ll N \frac{\log x}{\log y}.
$$
\end{lemma}

For $N \ge y \ge 2$, we use $\Psi(N,y)$ to denote  the set number of integers $n\le N$
whose prime divisors are at most  $y$.

The following bound on the cardinality $\# \Psi(N,y)$ well-known,   see, for example,~\cite[Part~III,
Theorem~5.1]{Ten}.

\begin{lemma}
\label{lem: bound Psi}
For any $N$ and $y$  with $N \ge y \ge 2$,  we have
$$
\# \Psi(N,y) \ll\ N\exp{\(-\,\frac{\log{N}}{2\log{y}}\)}.
$$
\end{lemma}

 \section{Sums of the divisor function over integers without small prime divisors}

For $N \ge y \ge 2$, we use $\Phi(N,y)$ to denote the set containing $1$ and all intergers $n\le N$
whose prime divisors are at least  $y$.

Suppose that $N \ge y \ge 2$. For an  integer $\nu \ge 1$, we set
$$
S_{\nu}(N;y)= \sum_{n \in \Phi(N,y)}\tau_{\nu}(n).
$$

Clearly $\tau_1(n) =1$ and so
\begin{equation}
\label{eq: Bound S1}
S_{1}(N,y) = \#  \Phi(N,y) \ll  \frac{N}{\log y},
\end{equation}
see, for example,~\cite[Part~III. Theorem~6.4]{Ten}.

\begin{lemma}
\label{lem:Sum Div}
For a fixed integer $\nu\ge 2$, for any $N$ and $y$  with $N \ge y \ge 2$,   we have
$$
S_{\nu}(N;y) \ll \frac{N (\log N)^{\nu-1}}{(\log y)^{\nu}}.
$$
\end{lemma}

\begin{proof} If $\sqrt{N}<y\le N$ then the desired estimate follows from prime number theorem, since in this case one has
$$
S_{\nu}(N;y) = \sum\limits_{\sqrt{N}<p\le N}\tau_{\nu}(p)\ll \frac{N}{\log{N}} \ll N \frac{(\log{N})^{\nu-1}}{(\log{y})^{\nu}}.
$$
Thus, in the below we assume that $2\le y\le \sqrt{N}$. 

We now establish the bound by induction on $\nu$.
The bound~\eqref{eq: Bound S1} provides the basis of induction
for $\nu=1$.

Now assume that the result holds for $S_{\nu-1}(N;y)$. We now derive it for
$\nu \ge 2$. First, we write
\begin{align*}
S_{\nu}(N;y)   & = \sum_{d\in \Phi(N,y)} \sum_{m\in \Phi(N/d,y)} \tau_{\nu-1}(m)
=  \sum_{d\in \Phi(N,y)} S_{\nu-1}(N/d;y) \\
& \le  \sum_{d\in \Phi(N/y,y)} S_{\nu-1}(N/d;y) +  S_{\nu-1}(N;y).
\end{align*}
Now, by the induction assumption, we obtain
\begin{align*}
S_{\nu}(N;y) & \le  \sum_{d\in \Phi(N/y,y)}  \frac{N (\log (N/d))^{k-2}}{d (\log y)^{k-1}} +   \frac{N (\log N)^{\nu-2}}{(\log y)^{k-1}}\\
& \le \frac{N (\log N)^{\nu-2}}{d (\log y)^{k-1}}   \sum_{d\in \Phi(N/y,y)}  \frac{1}{d} +   \frac{N (\log N)^{\nu-2}}{(\log y)^{k-1}}.
\end{align*}
Now, using~\eqref{eq: Bound S1}  by partial summation one easily derives that for any $Z\ge 1$
$$
 \sum_{d\in \Phi(Z,y)}  \frac{1}{d} \le  \frac{ \log Z}{ \log y}
$$
and the result follows.
\end{proof}


\section{Proof of Theorem~\ref{thm:Sum M}}

We remark that we can assume that $\gcd(a,p)=1$ as otherwise the result 
(in a much stronger form) 
follows from classical bound of Walfisz~\cite[Chapter~V, Section~5, Equation~(12)]{Wa} on sums of M{\"o}bius function. 

We can certainly assume that
$$
N < p
$$
as for larger values of $N$ the result of
Fouvry,  Kowalski and Michel~\cite[Theorem~1.7]{FKM1}
is stronger.

We fix some real $x$ and $y$ and for an integer $r \ge 0$ we
denote by $\cA_r(N;x,y)$ the set of positive integers $n \le N$
which have exactly $r$ prime divisors (counted with multiplicities) in the
 half-open interval $\cI = (x,y]$.

In particular,  the cardinality of $\cA_0(N;x,y)$ has been  estimated in  Lemma~\ref{lem:ANxy}.
Let $R$ be the largest value of $r$ for which $\cA_r(N;x,y) \ne \emptyset$.  In particular, we have the
trivial bound
\begin{equation}
\label{eq: R triv}
R \ll \log N.
\end{equation}

We now write
\begin{equation}
\label{eq: M Sr}
\M_{p}(\cK, N)  = \sum_{r =0}^R  U_r,
\end{equation}
where
$$
U_r = \sum_{n\in \cA_r(N;x,y)}\mu(n)\cK(n).
$$

We note that any trace  trace function $\cK(n)$ of the type we consider  is bounded pointwise by its conductor,
that is for a trace function $\cK(n)$ of bounded conductor we have 
\begin{equation}
\label{eq:Bound-B}
\cK(n) \ll 1. 
\end{equation}
In particular using~\eqref{eq:Bound-B} and estimating $A_0$ trivially as
$$
U_0 \ll \# \cA_0(N;x,y)
$$
and using Lemma~\ref{lem:ANxy}, we obtain
\begin{equation}
\label{eq: M Ur U0}
\M_{p}(\cK, N)  \ll  \sum_{r =1}^R  |U_r| +  N \frac{\log x}{\log y}.
\end{equation}

Clearly, every square-free integer $n \in \cA_r(N;x,y)$, which is not divisible by $\ell^2$
for a prime $\ell \in \cI$ has exactly $r$ representations as $n = \ell m$
with a prime $\ell \in \cI$  and integer $m \in \cA_{r-1}(N/\ell;x,y)$.
Trivially, the number of $n \le N$, which are  divisible by a square
of a prime $\ell \in \cI$, is at most
$$
\sum_{\ell \in \cI} \frac{N}{\ell^2} \le N \sum_{h\ge x} \frac{1}{h^2} \le N/x.
$$
%

Hence, for $r =1, \ldots, R$,
$$
U_r = \frac{1}{r} \sum_{\ell \in \cI}  \sum_{\substack{m\in \cA_{r-1}(N/\ell;x,y)\\ \gcd(\ell,m)=1}}\mu(\ell m)\cK(\ell m)  + O(N/x),
$$
where throughout the proof, $\ell$ always denotes a prime number.

Changing the order of summation and using the multiplicativity of the M{\"o}bius function, we now write
\begin{equation}
\label{eq:UrVr-1}
U_r = \frac{1}{r} V_r+ O(N/x),
\end{equation}
where, using $\mu(\ell)=-1$ we have
\begin{align*}
V_r  & =  \sum_{m\in \cA_{r-1}(N/x ;x,y)}   \mu(m)
 \sum_{\substack{\ell \in \cI\cap[1, N/m]\\ \gcd(\ell,m)=1}}
\mu(\ell)\cK(\ell m) \\
& =  -  \sum_{m\in \cA_{r-1}(N/x ;x,y)}   \mu(m)
 \sum_{\substack{\ell \in \cI\cap[1, N/m]\\ \gcd(\ell,m)=1}}
 \cK(\ell m)  .
\end{align*}

Let us define the integer  $K$ by the inequality
\begin{equation}
\label{eq: K def}
x2^K \le \fl{y}+1 < x2^{K+1}.
\end{equation}

We now partition the interval $\cI$ into at most $K+1 = O(\log y)$
intervals 
\begin{equation}
\label{eq: Ik}
\cI_k = (x_k,y_k]
\end{equation}
with
$$x_k = 2^{k} \rf{x} \mand y_k= \min\{2x_k, \fl{y}\},
\qquad k =0, \ldots, K.
$$

Thus
\begin{equation}
\label{eq: TrTkr}
\left| V_r\right|\le \sum_{k=0}^{K}  \left|V_{k,r}\right|,
\end{equation}
where
$$
V_{k,r} =  \sum_{m\in \cA_{r-1}(N/x_k ;x,y)}   \mu(m)
 \sum_{\substack{\ell \in \cI_k\cap[1, N/m]\\ \gcd(\ell,m)=1}}
\cK(\ell m) .
$$

 Clearly for each $m\in \cA_{r-1}(N/x_k  ;x,y)$, there are at most  $r-1$ primes $\ell \in \cI$
with $\gcd(\ell,m)> 1$.
Hence,  using~\eqref{eq:Bound-B}, at the cost of the error term $O(r N/x_k)$
we can  discard the co-primality condition $\gcd(\ell,m)=1$ and   write
\begin{equation}
\label{eq: VkrWkr-1}
V_{k,r} \ll   W_{k,r} + r N/x_k,
\end{equation}
where
$$
W_{k,r}=  \sum_{m\in \cA_{r-1}(N/x_k ;x,y)}   \left|
 \sum_{\ell \in \cI_k\cap[1, N/m]} \cK(\ell m)\right|.
$$
Now, by the Cauchy inequality, we obtain
\begin{equation}
\begin{split}
\label{eq:Wkr2}
W_{k,r}^2 \le  \# \cA_{r-1}&(N/x_k ;x,y)  \\
& \sum_{m\in \cA_{r-1}(N/x_k ;x,y)}  
\left|\sum_{\ell \in \cI_k\cap[1, N/m]} \cK(\ell m) \right|^2.
\end{split}
\end{equation}
We now  also use the trivial estimate
$$
\# \cA_{r-1}(N/x_k ;x,y)  \le N/x_k
$$
(more precise estimates are available but do not improve the final result)
and extend the summation to all positive integer $m \le N/x_k$, which yields
$$
W_{k,r}^2     \le Nx_k^{-1}  \sum_{m \le N/x_k }  
\left|\sum_{\ell \in \cI_k\cap[1, N/m]} \cK(\ell m) \right|^2.
$$
Finally, squaring out and changing the order of summation, we arrive to
\begin{equation}
\label{eq:Ukr2-1}
W_{k,r}^2     \le Nx_k^{-1}   \sum_{ \ell_1, \ell_2\in \cI_k}\,   \sum_{m \le N/\max\{x_k, \ell_1, \ell_2\}}  
 \cK( \ell_1 m)  \overline{\cK(\ell_2 m)}. 
\end{equation}
For at most $O(y_k)$ pairs $(\ell_1,\ell_2)$ with
$\ell_1 /\ell_2\in \cE_\fF$, where the set $\cE_\fF$ is as in 
Lemma~\ref{lem:sum K compl},  we estimate the inner sum trivially as $Nx_k^{-1}$,
For the remaining $O(y_k^2)$ pairs $(\ell_1,\ell_2)$ we recall that $N \le p$ and
apply  Corollary~\ref{cor:sum K compl}  with  $\nu =2$.
Therefore,
\begin{equation}
\begin{split}
\label{eq:Ukr2-2}
W_{k,r}^2  &   \ll Nx_k^{-1} \(y_k  Nx_k^{-1}    + y_k^2 p^{1/2}\log p\) \\
&   \ll  N^2   x_k^{-1}  + Np^{1/2}x_k \log p.
\end{split}
\end{equation}
Recalling the definition $x_k = 2^k \rf{x}$ we obtain
\begin{equation}
\label{eq: Ukr bound}
W_{k,r}   \ll  N  x^{-1/2} 2^{-k/2} + N^{1/2} p^{1/4}  x^{1/2}  2^{k/2}  (\log p)^{1/2}.
\end{equation}
Substituting~\eqref{eq: Ukr bound}  in~\eqref{eq: VkrWkr-1}, we obtain
$$
V_{k,r}   \ll  N x^{-1/2} 2^{-k/2} + N^{1/2} p^{1/4}  x^{1/2}  2^{k/2}  (\log p)^{1/2}+  r N x^{-1} 2^{-k}.
$$
Thus, substituting this bound in~\eqref{eq: TrTkr} we derive
\begin{align*}
V_r& \ll \sum_{k=0}^{K}
 \( N x^{-1/2} 2^{-k/2}   +N^{1/2} p^{1/4}  x^{1/2}  2^{k/2}  (\log p)^{1/2}+ r N x^{-1} 2^{-k}\)\\
 & \ll  N   x^{-1/2} + N^{1/2}  p^{1/4} x^{1/2}  2^{K/2}  (\log p)^{1/2} +  r N/ x .
\end{align*}
Therefore, recalling the definition of $K$ given by~\eqref{eq: K def}, we obtain
\begin{equation}
\label{eq: Vr bound}
V_r  \ll  N   x^{-1/2} + N^{1/2}  p^{1/4} y^{1/2}(\log p)^{1/2}+   r N /x .
\end{equation}
Using the bound~\eqref{eq: Vr bound} in~\eqref{eq:UrVr-1} and then in~\eqref{eq: M Ur U0} we see that
 \begin{align*}
&\M_{p}(\cK, N) \\
 & \quad  \ll  \sum_{r =1}^R  \frac{1}{r}\( N   x^{-1/2} + N^{1/2}  p^{1/4} y^{1/2}(\log p)^{1/2}+ r N/ x\) +  N \frac{\log x}{\log y}\\
&\quad  \ll  N  x^{-1/2}  \log R + N^{1/2}  p^{1/4} y^{1/2}(\log p)^{1/2} \log R +  N R/ x+ N \frac{\log x}{\log y} .
\end{align*}
We now choose $x = (\log p)^4$, $y = p^{\varepsilon/3}$ and   recall that by~\eqref{eq: R triv}
we have $\log R \ll\log \log p$.
 The result now follows.

\section{Proof of Theorem~\ref{thm:Sum T}}

We fix some real $x$ and $y$ and define the sets  $\cA_r(N;x,y)$ and the integer $R$
as in the proof of Theorem~\ref{thm:Sum M}. In particular, the bound~\eqref{eq: R triv}
still holds.

Now, instead of~\eqref{eq: M Sr}, we have
$$
\T_{p,\nu}(\cK,N)= \sum_{r =0}^R  U_r,
$$
where we now define
$$
U_{r}=\sum_{n\in \cA_r(N;x,y)}\tau_\nu(n)\cK(n).
$$

As in the the proof of  Theorem~\ref{thm:Sum M}, we  estimate the sums $U_{0}$ and $U_{r}$ with $r\ge 1$, separately.
We note that
now estimating $U_0$ takes slightly more care than the corresponding bound on $U_0$ in
the proof of  Theorem~\ref{thm:Sum M}.

We observe that any $n\in \cA_0(N;x,y)$ can be uniquely expressed in the form
$n = uv$ or $n = u$ where $u\in \Psi(n,x)$ has no prime divisors greater than $x$ and
$v \in\Phi^*(n,y)$, where
$$
\Phi^*(n,y) = \Phi(n,y) \setminus \{1\}.
$$
Recalling~\eqref{eq:Bound-B}, we see that
$$
|U_{0}| \ll U_{0,1}+U_{0,2},   
$$
where
$$
U_{0,1} = 2\sum_{u\in \Psi(N,x)}\tau_\nu(u) \mand U_{0,2}= \sum_{\substack{u\in \Psi(N,x)\\ v \in\Phi^*(N,y) \\ uv\le N}}
\tau_\nu(uv).
$$

We estimate $U_{0,1}$ rather crudely. Namely, using the bound
\begin{equation}
\label{eq:sum tau-k-2}
\sum_{n\le z}\tau_\nu^{2}(n) \ll z (\log z)^{\nu^2-1},
\end{equation}
which holds for anty real $z \ge 2$,
see~\cite[Equation~(1.80)]{IwKow}, by the Cauchy inequality and
Lemma~\ref{lem: bound Psi} we see that
\begin{equation}
\begin{split}
\label{eq:U01}
U_{0,1} & \le\( \# \Psi(N,x)\sum_{n\le N}\tau_\nu^{2}(n)\)^{1/2} \\
& \ll N\exp{\(-\,\frac{\log{N}}{4\log{x}}\)}(\log{N})^{(\nu^2-1)/2}.
\end{split}
\end{equation}

Next, recalling the multiplicativity of $\tau(n)$ and using the notations and the bound of
Lemma~\ref{lem:Sum Div}, we find
\begin{align*}
U_{0,2} & \le\sum_{u\in \Psi(N/y,x)}\tau_\nu(u)\sum_{v \in\Phi^*(N/u,y)}\tau_\nu(v) \le\sum_{u\in \Psi(N/y,x)}
\tau_\nu(u) S_\nu\(Nu^{-1};y\) \\
&\ll \frac{N (\log N)^{\nu-1}}{(\log y)^{\nu}}\sum_{u\le N/y}\frac{\tau_\nu(u)}{u} \\
& \ll  \frac{N (\log N)^{\nu-1}}{(\log y)^{\nu}}\prod_{\substack{\ell \le x\\\ell~\text{prime}}} \sum_{j=0}^\infty \frac{1}{\ell^j} \binom{j+\nu-1}{j}. 
\end{align*}
Clearly
$$
 \sum_{j=0}^\infty \frac{1}{\ell^j} \binom{j+\nu-1}{j} = \(\sum_{i=0}^\infty \frac{1}{\ell^i} \)^\nu
 = \(1-\frac{1}{\ell}\)^{-\nu}.
 $$
 Hence, by the Mertens theorem,  see~\cite[Equation~(2.16)]{IwKow}
\begin{equation}
\begin{split}
\label{eq:U02}
U_{0,2} & \ll  \frac{N (\log N)^{\nu-1}}{(\log y)^{\nu}}\prod_{\substack{\ell \le x\\\ell~\text{prime}}} \(1-\frac{1}{\ell}\)^{-\nu}\\
& \ll \frac{N (\log N)^{\nu-1}(\log x)^{\nu}}{(\log y)^{\nu}}. 
\end{split}
\end{equation}

Therefore, combining~\eqref{eq:U01} and~\eqref{eq:U02}, we see that
instead of~\eqref{eq: M Ur U0} we now have
\begin{equation}
\begin{split}
\label{eq: T Ur U0}
\T_{p,\nu}(\cK,N)&\ll \sum_{r =1}^R  |U_r| +   N\exp{\(-\,\frac{\log{N}}{4\log{x}}\)}(\log{N})^{(\nu^2-1)/2}\\
&  \qquad  \qquad  \qquad +\frac{N (\log N)^{\nu-1}(\log x)^{\nu}}{(\log y)^{\nu}}.
\end{split}
\end{equation}

Now, let $R \ge r\ge 1$ so that $\cA_r(N;x,y)\ne \emptyset$. Then writing $n = \ell m\in \cA_r(N;x,y)$  with a prime $\ell \in \cI$  and integer $m \in \cA_{r}(N/\ell;x,y)$ we have
\begin{equation}
\label{eq: T Ur Er}
U_{r} = \frac{1}r V_{r}+ E_r, 
\end{equation}
where
$$
V_{r}= \sum_{x<\ell\le y}\,\sum_{\substack{m\in \cA_{r-1}(N/\ell;x,y) \\
 \gcd(m,\ell)=1}}\tau_\nu(\ell m)\cK(\ell m), 
 $$
and $E_r$ is the contribution of $m$ divisible by
$\ell $ for some prime $\ell\in \cI$ and thus $U_{r,2}$ is the contribution
from other $n$.
We estimate $E_r$ rather crudely. Namely, using the bound~\eqref{eq:sum tau-k}. we obtain
\begin{align*}
|E_{r}| &\le 2\sum_{x<\ell\le y}\sum_{\substack{m \le N/\ell\\\ell \mid m}}\tau_\nu(\ell m)
\le 2\sum_{x<\ell\le y}\sum_ {m \le N/\ell^2}\tau_\nu(\ell^{2}m) \\
&  \le 2\sum_{x<\ell\le y}\tau_\nu(\ell^{2})\sum_{m\le N/\ell^{2}}\tau_\nu(m) \ll N(\log N)^{\nu-1}\sum_{\ell>x}\frac{1}{\ell^{2}}\\
& \ll  Nx^{-1}(\log N)^{\nu-1} .
\end{align*}
Therefore, we can rewrite~\eqref{eq: T Ur Er}
as
\begin{equation}
\label{eq:UrVr-2}
U_{r} = \frac{1}{r}   V_{r}+ O\(  Nx^{-1}(\log N)^{\nu-1}\)
\end{equation}
which is an analogue of~\eqref{eq:UrVr-1}.

Furthermore, any $n\in \cA_r(N;x,y)$ from the sum $U_{r,1}$ has exactly $r$ representations of the form $n = \ell m$ where $\ell\in \cI$
and $m\in \cA_{r-1}(N;x,y)$ with $\gcd(m,\ell)=1$. Therefore, using $\tau_\nu(\ell) =\nu$, we obtain
\begin{align*}
V_{r}=\frac{1}{r}\sum_{x<\ell\le y}\,\sum_{\substack{m\in \cA_{r-1}(N/\ell;x,y) \\
 \gcd(m,\ell)=1}}\tau_\nu(\ell m)\cK(\ell m)\\
= \frac{\nu}{r}\sum_{x<\ell\le y}\,\sum_{\substack{m\in \cA_{r-1}(N/\ell;x,y) \\ \gcd(m,\ell)=1}}\tau_\nu(m)\cK(\ell m).
\end{align*}
Thus, defining the intervals $\cI_k$, $k = 0, \ldots, K$ as in~\eqref{eq: Ik}, we have the inequality
$$
\left| V_r\right|\le \sum_{k=0}^{K}  \left|V_{k,r}\right|,
$$
where
\begin{align*}
V_{k,r}&=    \sum_{m\in \cA_{r-1}(N/x_k ;x,y)}   \tau_\nu(m)
 \sum_{\substack{\ell \in \cI_k\cap[1, N/m]\\ \gcd(\ell,m)=1}}
\cK(\ell m) .
\end{align*}

If we drop the condition $\gcd(m,\ell)=1$ if the above inner sum then as in~\eqref{eq: VkrWkr-1}
we see that it introduces an  error that by  absolute value does not exceed the quantity
$$
2(r-1)    \sum_{m\in \cA_{r-1}(N/x_k ;x,y)}   \tau_\nu(m) \ll r\frac{N(\log N)^{\nu-1}}{x_k}.
$$
Thus, we have the following analogue of~\eqref{eq: VkrWkr-1}
\begin{equation}
\label{eq: VkrWkr-2}
V_{k,r}= W_{k,r} + O\(r\frac{N(\log N)^{\nu-1}}{x_k}\),
\end{equation}
where
$$
W_{k,r}=  \sum_{m\in \cA_{r-1}(N/x_k ;x,y)}  \tau_\nu(m) \left|
 \sum_{\ell \in \cI_k\cap[1, N/m]} \cK(\ell m)\right|.
$$
Now, by the Cauchy inequality and also recalling~\eqref{eq:sum tau-k-2}, we obtain
$$
W_{k,r}^2 \le   \frac{N (\log{N})^{\nu^2-1} }{x_{k}}   
  \quad  \sum_{m\in \cA_{r-1}(N/x_k ;x,y)} \left|
 \sum_{\ell \in \cI_k\cap[1, N/m]} \cK(\ell m)\right|^2 
$$
instead of~\eqref{eq:Wkr2}. We now proceed exactly as in the proof of Theorem~\ref{thm:Sum M}
except that we have an extra factor of  $(\log{N})^{\nu^2-1}$ in analogues of the bounds~\eqref{eq:Ukr2-1},
\eqref{eq:Ukr2-2} and~\eqref{eq: Ukr bound}.  Hence recalling~\eqref{eq: VkrWkr-2}
  and using very crude estimates
$$
\nu -1 < (\nu^2-1)/2 < \nu^2-1 < \nu^2 
$$
(as using more accurate bounds does not change the final result), 
instead of~\eqref{eq: Vr bound} we obtain
$$
V_{r}   \ll  \(N   x^{-1/2} + N^{1/2}  p^{1/4} y^{1/2}(\log p)^{1/2}+   r N x^{-1}\)(\log N)^{\nu^2} .
$$
Using this bound in~\eqref{eq:UrVr-2} and then in~\eqref {eq: T Ur U0}, we see that
 \begin{align*}
&\T_{p,\nu}(\cK,N)\\
 & \quad  \ll  (\log N)^{\nu^2}\sum_{r =1}^R  \frac{1}{r}\( N   x^{-1/2} + N^{1/2}  p^{1/4} y^{1/2}(\log p)^{1/2}+ r N/x\) \\
&  \qquad \quad +   N\exp{\(-\,\frac{\log{N}}{4\log{x}}\)}(\log{N})^{(\nu^2-1)/2}
+\frac{N (\log N)^{\nu-1}(\log x)^{\nu}}{(\log y)^{\nu}}\\
&\quad  \ll  \(N  x^{-1/2}  \log R + N^{1/2}  p^{1/4} y^{1/2}(\log p)^{1/2} \log R +  N R/x\)(\log N)^{\nu^2}\\
&  \qquad \quad +   N\exp{\(-\,\frac{\log{N}}{4\log{x}}\)}(\log{N})^{(\nu^2-1)/2} +\frac{N (\log N)^{\nu-1}(\log x)^{\nu}}{(\log y)^{\nu}}.
\end{align*}
We now choose  $x = (\log p)^{2(\nu^2+1)}$, $y = p^{\varepsilon/3}$ and   recall that by~\eqref{eq: R triv}
we have $\log R \ll\log \log p$.
The result now follows.

 \section{Comments}

Examining~\cite[Proposition~6.2 and Theorem~6.3]{FKM1}, one can easily 
see that the dependence of implied constants on conductor of $\cK$ in the 
assumptions of Theorems~\ref{thm:Sum M} and~\ref{thm:Sum T} is polynomial.

It is easy to see that our approach also applies  to the sums
$$
  \sum_{n\le N}  \mu(n)\e\left(a g^n/q \right)
\mand
  \sum_{n\le N}  \mu(n)\chi\left(ag^n + 1\right)
$$
with some integer $g$ of multiplicative order $t$ modulo a  $q$, see~\cite[Theorem~5.1]{BCFS}.
Using the techniques and results from~\cite{BFGS,Bourg} one is likely to be able to
improve~\cite[Theorem~5.1]{BCFS}, however our approach leads to nontrivial bounds
in a wider range of parameters $N$ and $t$. 

One can also obtain similar results for sums with the divisor function $\tau_\nu(n)$ instead of the
\Mob function and in fact with many other multiplicative functions.

We also note that  a wide class of trace functions, including Kloosterman sums, a very
broad extension of  Lemma~\ref{lem:sum K compl} has been given in~\cite{FKM2}. To formulate 
this we consider the actions 
$$
\gamma(n) = \frac{an + b}{cn + d}
$$
of matrices
$$
\gamma =\begin{pmatrix} a & b\\c & d \end{pmatrix} \in \mathrm{PGL}_2(\F_p). 
$$
Then for a wide class of  of trace functions $\cK(n)$, under some natural {\it nodegenerosity\/}
conditions of the matrices $\gamma_1, \ldots, \gamma_m$ we have 
\begin{equation}
\label{eq: PGL Bound}
 \sum_{n=1}^p \prod_{j=1}^{m}  \cK\(\gamma_j(n)\)  \e(hn/p) \ll p^{1/2} ,
\end{equation}
see~\cite[Corollary]{FKM2}. 
In the most interesting case when $\cK(n)$ is given by  Kloosterman sums $\cK(n) = p^{-(s-1)/2} K_{s,p}(n)$
these nodegenerosity conditions reduce to the request that at least one matrix $\gamma$ is the sequence 
$\gamma_1, \ldots, \gamma_m$ appears an odd number of times.

Combining~\eqref{eq: PGL Bound} (which we actually need only for linear transformations  $n \mapsto an + b$) with the argument of~\cite{OstShp}
one can obtain nontrivial bounds on sums of trace functions $\cK(n)$ over integers $n$ with a fix sum 
of binary digits. More precisely, let $\sigma(n)$  denote the sum of binary digits of $n$. 
For any integers $0\le s\le r$,  we define  $\cG_s(r)$ as the set of integers with $r$ binary digits
 such that the sum of the digits is equal to $s$, that is, 
$$\cG_s(r)=\{0 \leq n< 2^r  \mid \sigma(n)=s\} \mand \#  \cG_s(r) = \binom{r}{s}.$$
Then, the bound~\eqref{eq: PGL Bound}   implies an  analogue  of~\cite[Theorems~1 and~2]{OstShp} 
for a wide class of  trace functions $\cK(n)$
In particular, as in~\cite{OstShp} we see that  if $2^r = p^{1+o(1)}$ then for any $\delta> 0$ there exists some $\eta>0$ such that 
for $r/2 \ge s \ge (\rho_0 + \delta) r$ 
we have 
\begin{equation}
\label{eq: DigitBound}
\sum_{n \in \cG_s(r)} \cK(n) \ll \binom{r}{s}^{1-\eta}, 
\end{equation}
where $\rho_0 = 0.11002786\ldots$  is the root of the equation 
$$
H(\vartheta) = 1/2 \qquad 0 \le \vartheta \le 1/2, 
$$
with the {\it binary entropy function\/}
$$
H(\gamma) = \frac{- \gamma \log \gamma  -  (1-\gamma) 
\log (1-\gamma)}{\log 2}  . 
$$
In particular,  the bound~\eqref{eq: DigitBound} holds for  sums with  Kloosterman sums $\cK(n) = p^{-(s-1)/2} K_{s,p}(n)$. 

It is also interesting to consider sums of trace functions over integers with other digit restrictions.for example, for integers
with fixed binary digits at $s$ prescribed positions, see~\cite{DES} for some   relevant results. 

 \section*{Acknowledgement}

The authors are very grateful to  Emmanuel Kowalski and Philippe Michel 
for their comments and suggestions. The authors also would like to thank
 G{\'e}rald Tenenbaum for outlining the argument
of the proof of Lemma~\ref{lem:Sum Div} which significantly simplified
the original treatment.

During the preparation of this work the first author was supported by
the Russian Science Foundation Grant 14-11-00433 and
the second author was supported in part by
the  Australian Research Council  Grants DP170100786 and DP180100201.


\begin{thebibliography}{9999}


\bibitem{AnDu}
N. Andersen and W. Duke,
 `Modular invariants for real quadratic fields and Kloosterman sums',
{\it Preprint\/}, 2018  (available
from \url{http://arxiv.org/abs/1801.08174}).

\bibitem{BCFS}
W. Banks, A. Conflitti, J. B. Friedlander and I. E. Shparlinski, `Exponential sums over Mersenne numbers', {\it  Compositio Math.\/}, {\bf 140} (2004), 15--30.

\bibitem{BFGS}
W. Banks, A. Conflitti, J. B. Friedlander and I. E. Shparlinski, `Exponential   and character sums with  Mersenne numbers', {\it J. Aust. Math. Soc.\/}, {\bf  92} (2012), 1--13.

\bibitem{BlMil} V. Blomer and D. Mili{\'c}evi{\'c},
`Kloosterman sums in residue classes',
{\it  J. Eur. Math. Soc.\/}, {\bf 17} (2015), 51--69.

\bibitem{BFKMM1} V. Blomer, {\'E}. Fouvry, E. Kowalski, P. Michel and D. Mili{\'c}evi{\'c},
 `On moments of twisted $L$-functions',
 {\it Amer. J.  Math.\/}, {\bf 139} (2017), 707--768.

\bibitem{BFKMM2} V. Blomer, {\'E}. Fouvry, E. Kowalski, P. Michel and D. Mili{\'c}evi{\'c},
 `Some applications of smooth bilinear forms with Kloosterman sums',
  {\it Proc. Steklov Math. Inst.\/},  {\bf 296} (2017),  18--29.

\bibitem{Bourg} J. Bourgain,
`Estimates  on exponential sums related to Diffie-Hellman
distributions', {\it  Geom. and Funct. Anal.\/},
{\bf 15} (2005), 1--34.
%
  \bibitem{BSZ} J. Bourgain,  P. Sarnak  and T. Ziegler,
`Disjointness of \Mob  from horocycle flow',
{\it From Fourier Analysis and Number Theory to Radon Transforms
and Geometry\/},
Devel.  Math.,  {\bf 28}, Springer, New york, 2013,  67--83.

  \bibitem{CLS} T. Cochrane, C. L, Liu and  Z. y.  Zheng,
`Upper bounds on character sums with rational function entries',
{\it Acta Math. Sin. (Engl. Ser.)\/},  {\bf  19} (2003),  327--338.

 \bibitem{CoZh} T. Cochrane and Z. y. Zheng, `Pure and mixed exponential
sums',  {\it Acta Arith.\/}, {\bf   91} (1999), 249--278.

\bibitem{Drap} S. Drappeau, `Sums of Kloosterman sums in arithmetic progressions,
and the error term in the dispersion method',
{\it  Proc. Lond. Math. Soc.\/}, {\bf 114} (2017),  684--732.

\bibitem{DES} R. Dietmann, C. Elsholtz and I. E.~Shparlinski,
`Prescribing the binary digits of squarefree numbers and quadratic residues',
{\it Trans. Amer. Math. Soc.\/}, {\bf  369}  (2017), 8369--8388.

\bibitem{FKM1} {\'E}. Fouvry,  E. Kowalski and
  P.  Michel, `Algebraic trace functions over the primes',
{\it Duke Math. J.\/},  {\bf 163} (2014),  1683--1736.

\bibitem{FKM1.5} {\'E}. Fouvry,  E. Kowalski and
  P.  Michel, `Trace functions over finite fields and their applications',
{\it  Colloquium De Giorgi 2013 and 2014\/}, Publ.    Scuola Normale Superiore, 
vol.~5. Edizioni della Normale, Pisa, 2015.

\bibitem{FKM2} {\'E}. Fouvry,  E. Kowalski and
  P.  Michel, `A study in sums of products',
{\it Phil. Trans. R. Soc., Ser. A\/},  {\bf  373} (2015),  20140309.



\bibitem{FKMRRS} {\'E}. Fouvry,  E. Kowalski,
  P.  Michel,  C. S. Raju, J. Rivat and  K. Soundararajan,
 `On short sums of trace functions',
 {\it Annales de l'Institut Fourier\/},  {\bf 167}  (2017),  423--449.


\bibitem{FMRS} {\'E}. Fouvry,
  P.  Michel,  J. Rivat and A. S{\'a}rk{\"o}zy,
`On the pseudorandomness of the signs of Kloosterman sums',
{\it J. Aust. Math. Soc.\/},  {\bf  77} (2004),  425--436.

\bibitem{FrIw} J. Friedlander and H. Iwaniec,
{\it Opera de Cribro\/}, Colloq. Publ. {\bf 57}
Amer. Math. So., Providence, RI, 2010.

\bibitem{GJK}
K. Gong, C. Jia and M. A. Korolev,
`Shifted character sums with multiplicative coefficients, II',
{\it J. Number Theory\/},  {\bf 178} (2017), 31--39.

\bibitem{IwKow} H. Iwaniec and E. Kowalski,
{\it Analytic number theory\/}, Amer.  Math.  Soc.,
Providence, RI, 2004.




\bibitem{Kir} E. M.  Kiral, `Opposite-sign Kloosterman sum zeta function',
{\it Mathematika\/},  {\bf 62} (2016), 406--429.

\bibitem{Kor1} M. A. Korolev,
`Short Kloosterman sums with weights', {\it Math. Notes\/}, {\bf 88}  (2010), 374--385;
translated from {\it Matem. Zametrki\/} (in Russian).

\bibitem{Kor2} M. A. Korolev,
`On Kloosterman sums with multiplicative coefficients',
{\it Izv. RAN. Ser. Matem.\/}, {\bf 82} (2018), (to appear).

\bibitem{KMS1} E. Kowalski,  P.  Michel and  W. Sawin,
 `Bilinear forms with Kloosterman sums and applications',
{\it Annals  Math.\/}, {\bf 186} (2017), 413--500.

\bibitem{KMS2} E. Kowalski,  P.  Michel and  W. Sawin,
 `Bilinear forms with generalized Kloosterman sums',
{\it Preprint\/}, 2018  (available
from \url{http://arxiv.org/abs/1802.09849}).

\bibitem{Kuz} N. V. Kuznetsov, `The Petersson conjecture for cusp forms of weight zero
and the Linnik conjecture. Sums of Kloosterman sums', {\it Math. USSR-Sb.\/},  {\bf 39} (1981), 299--342.

\bibitem{Lin}
y. V. Linnik, `Additive problems and eigenvalues of the modular operators,'
{\it Proc. Internat. Congr. Mathematicians (Stockholm, 1962)\/}, Inst. Mittag--Leffler, Djursholm, 1963,  270--284.

\bibitem{LSZ1} K. Liu, I. E. Shparlinski and T. P. Zhang,
 `Divisor problem in arithmetic progressions modulo a prime power',
{\it Adv. Math.\/}, {\bf 325} (2018), 459--481.

\bibitem{LSZ2} K. Liu, I. E. Shparlinski and T. P. Zhang,
 `Cancellations between Kloosterman sums modulo a prime power with prime arguments',
{\it Preprint\/}, 2016  (available
from \url{http://arxiv.org/abs/1612.05905}).

\bibitem{OstShp} A. Ostafe and I.~E.~Shparlinski,  
`Multiplicative character sums and products of sparse integers in residue classes', 
{\it Period. Math. Hungarica\/},  {\bf 64} (2012), 247--255.

\bibitem{Sarn}   P. Sarnak,
 `M{\"o}bius  randomness and dynamics',
{\it Not. South Afr. Math. Soc.\/}, {\bf 43} (2012), 89--97.

\bibitem{SarTsim}
P. Sarnak and J. Tsimerman, `On Linnik and Selberg's conjecture about sums of Kloosterman sums',
{\it Algebra, Arithmetic, and Geometry: in Honor of Yu. I. Manin, vol. II\/}, Progress in Mathematics vol.~270,
Birkha{\"u}ser, Boston, MA, 2009, 619--635.

\bibitem{Shp} I. E. Shparlinski, `Bilinear forms with Kloosterman and Gauss sums',
{\it Trans. Amer. Math. Soc.\/}, (to appear).

\bibitem{ShpZha} I. E. Shparlinski and T. P. Zhang,
`Cancellations amongst Kloosterman sums',
{\it Acta Arith.\/}, {\bf  176} (2016), 201--210.

\bibitem{Ten} G.~Tenenbaum, {\it Introduction to analytic and
probabilistic number theory\/}, Grad. Studies Math., vol.~163, Amer. Math. Soc., 2015.

\bibitem{WuXi} J. Wu and   P. Xi,
 `Arithmetic exponent pairs for algebraic trace functions and applications',
{\it Preprint\/}, 2016  (available
from \url{http://arxiv.org/abs/1603.07060}).

\bibitem{Xi-FKM}  P. Xi (with an appendix by  {\'E}. Fouvry,  E. Kowalski and
  P.  Michel),
 `Large sieve inequalities for algebraic trace functions',
{\it Intern. Math.  Res. Notices\/}, {\bf 2017} (2017),  4840--4881.

\bibitem{Wa} A. Walfisz, `Weylsche Exponentialsummen in der neueren Zahlentheorie', Leipzig: B.G. Teubner, 1963.


\end{thebibliography}
\end{document}